\documentclass[12pt]{article}

\usepackage{amsmath}  
\usepackage{latexsym}
\usepackage{amssymb}
\usepackage{ifthen}
\usepackage{tikz}
\usepackage{amsthm}
\usepackage{yfonts}
\usepackage{soul}

\usepackage{hyperref}

\newcommand{\junk}[1]{}

\newcommand{\formal}[1]{\ensuremath{\textsf{#1}}}
\newcommand{\dword}[1]{\textit{\bf #1}}

\newcommand{\burn}{\ensuremath{\formal{b}}}
\newcommand{\cover}[1]{\ensuremath{\langle #1 \rangle}}
\newcommand{\cocoon}[1]{\ensuremath{\langle \langle #1 \rangle \rangle}}

\newcommand{\subtree}[1]{\ensuremath{T[#1]}}
\newcommand{\bigT}{\ensuremath{\widetilde{T}}}  
\newcommand{\rad}{\ensuremath{\formal{rad}}}
\newcommand{\nv}{\ensuremath{n}} 
\newcommand{\h}{\ensuremath{h}} 
\newcommand{\te}{\ell}
\newcommand{\Q}{q}

\newtheorem{defn}{Definition}[section]
\newtheorem{thm}[defn]{Theorem}
\newtheorem{prop}[defn]{Proposition}

\newtheorem{lemma}[defn]{Lemma}

\newtheorem{observation}[defn]{Observation}

\newtheorem{remark}[defn]{Remark}

\newtheorem{quest}[defn]{Question}

\title {Graph Burning On Large $p$-Caterpillars}
\author{Danielle Cox \and M.E. Messinger \and Kerry Ojakian}

\begin{document}
\maketitle

\begin{abstract} 

Graph burning models the spread of information or contagion in a graph.  At each time step, two events occur: neighbours of already burned vertices become burned, and a new vertex is chosen to be burned.  The big conjecture is known as the {\it burning number conjecture}: for any connected graph on $n$ vertices, all $n$ vertices can be burned after at most $\lceil \sqrt{n}\ \rceil$ time steps.  It is well-known that to prove the conjecture, it suffices to prove it for trees.  We prove the conjecture for sufficiently large $p$-caterpillars. 

\end{abstract}

\section{Introduction}

The process called \emph{graph burning} models the spread of information or contagion through a graph. Graph burning was introduced by Bonato, Janssen, and Roshanbin~\cite{bon2014,bon2016}, who describe a process in which a fire spreads through a simple undirected graph. There are two possible states for a vertex: burned or unburned; and initially, each vertex is unburned.   During the first round, a vertex is chosen to be burned.  During each subsequent round, two events occur: every unburned neighbour of a burned vertex becomes burned; and an unburned vertex is selected to be burned, provided such an unburned vertex exists.  The vertex selected at each round is called a \dword{source}. The process continues until every vertex in the graph is burned. The central question asks how quickly the fire propagates through the graph. The 
\dword{burning number} of a graph $G$, denoted $\burn(G)$, is the minimum number of rounds needed to burn every vertex of $G$.

The burning number has been studied for a variety of classes of graphs, including random graphs, theta graphs, generalized Petersen graphs, path-forests, hypercubes, graph products, and some trees.  See the survey by Bonato~\cite{Bonato} for more information about the burning number for various classes of graphs.  Although the burning number can be determined in polynomial time for cographs and split graphs~\cite{Kare}, the associated decision problem is {\bf NP}-complete in general.  It remains {\bf NP}-complete for trees with maximum degree 3, spider graphs, caterpillars of maximum degree 3, interval graphs, connected proper interval graphs, connected cubic graphs, permutation graphs, and disk graphs; see~\cite{Dalu,Bessy,Gupta,cat2019,Liu}.  For a connected graph on $n$ vertices, a central question is whether the burning number is bounded by $\lceil \sqrt{n}\ \rceil$:

\bigskip

\noindent {\bf Burning number conjecture (BNC)}.~\cite{bon2016} For a connected graph $G$ on $n$ vertices, $\burn(G) \leq \lceil \sqrt{n}\ \rceil$.\\

In~\cite{bon2016}, it was observed that if $H$ is a spanning subgraph of connected graph $G$, then $\burn(G) \leq \burn(H)$.  Thus, to prove the conjecture in general, it suffices to prove the conjecture for trees.  See~\cite{multi2022,Bonato,Norin} for asymptotic bounds on the burning number.  The conjecture has been proven to be true for paths~\cite{bon2016}, spiders~\cite{BonatoLidbetter,Das}, trees whose non-leaf vertices have degree at least 4~\cite{Omar}, trees whose non-leaf vertices have degree at least 3 (on at least $81$ vertices)~\cite{Omar}, $1$-caterpillars~\cite{cat2019,Liu}, and $2$-caterpillars~\cite{cat2019}.  For $p \geq 1$, a \dword{$p$-caterpillar} is a tree that contains a maximal path $P$ (called the \dword{spine}) such that every vertex is distance at most $p$ to $P$. 
The BNC is unresolved for $p$-caterpillars with $p \geq 3$.  
We state our main result next
(the proof appears in Section~\ref{sec:pcats}).

\begin{thm}\label{thm:catp} The BNC holds for any $p$-caterpillar on at least $16(4p^3 + 2p^2 + 4p)^2p^2$ vertices.
\end{thm}

\section{Proof Outline}

We describe an alternate approach to graph burning (developed in \cite{bon2016}), and provide some intuition on how we use this approach to prove our main theorem.

\begin{defn}
Let $G$ be a graph.  A {\bf ball of radius} {\boldmath{$r$}}  in $G$ is a subset of vertices of $V(G)$ containing a vertex $x \in V(G)$ and all vertices within distance $r$ of $x$, for some integer $r \geq 0$.  Such a ball is said to be {\bf centered at} {\boldmath{$x$}}.  If $B$ is a ball, we optionally superscript, writing $B^{[r]}$ to indicate that the ball has radius $r$, or write that $\rad(B) = r$.
\end{defn}
In graph-theoretic literature, the set of vertices in a ball of radius $r$ centered at vertex $x$  is sometimes referred to the $r^{th}$ 
\emph{closed neighbourhood} of $x$.  

\begin{defn}
Let $G$ be a graph.  A \dword{cover} of $G$ is a collection of balls such that every vertex is in at least one ball.  A cover is \dword{distinct} if no two balls have the same radius.
\end{defn}

\noindent \emph{Henceforth, all covers are assumed to be distinct.  Furthermore, all graphs considered in this paper are assumed to be $p$-caterpillars.} We will also always be working with \dword{spine covers} of $p$-caterpillars: covers where each ball is centered on the spine.

An example of a cover of the tree $T^*$  (drawn at the right in Figure~\ref{fig:bad_idea}) is the following: a radius $0$ ball centered at $v_1$, a 
radius $1$ ball centered at $v_8$, and a radius 2 ball centered at $v_4$.
This example corresponds to the graph burning process in which 3 sources
are chosen: first $v_4$, then $v_8$, then $v_1$; the balls correspond to the 
vertices burned by the respective sources.  So we can see that $\burn(T^*) \leq 3$.  
In general for a graph $G$, $\burn(G) = r+1$, where $r$ is the smallest
integer such that
$r$ is the radius of the largest ball in a cover of $G$.

We refer to a vertex on the spine $P$ with degree greater than 2 as a \dword{root}.  
In the subgraph induced by the deletion of the edges of $P$, each non-isolated spine vertex (i.e. a root) can be viewed as the root to a subtree; we denote such a subtree with root $x$ by \subtree{x}.  Note that each root has a subtree of height at most $p$. For example, the left tree in Figure~\ref{fig:bad_idea} is a 1-caterpillar whose spine can be taken to be the path from $v_1$ to $v_8$, and thus has roots $v_3, v_4, v_5$, and $v_6$, where each is the root of a height 1 subtree.

\begin{figure}[htbp] 
\[ \includegraphics[width=0.9\textwidth]{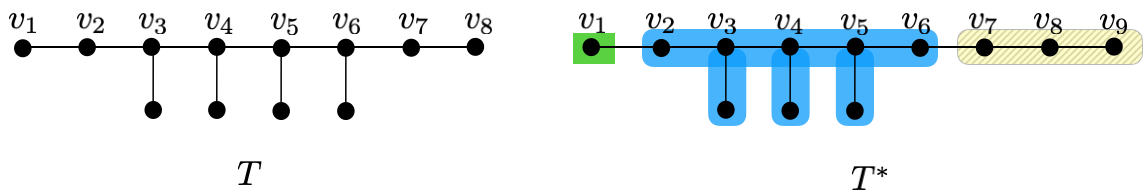} \] 

\caption{Caterpillars $T$ and $T^*$ for which $\burn(T)>\burn(T^*)$.}

\label{fig:bad_idea}
\end{figure}

\begin{figure}[h]
\[ \includegraphics[width=0.95\textwidth]{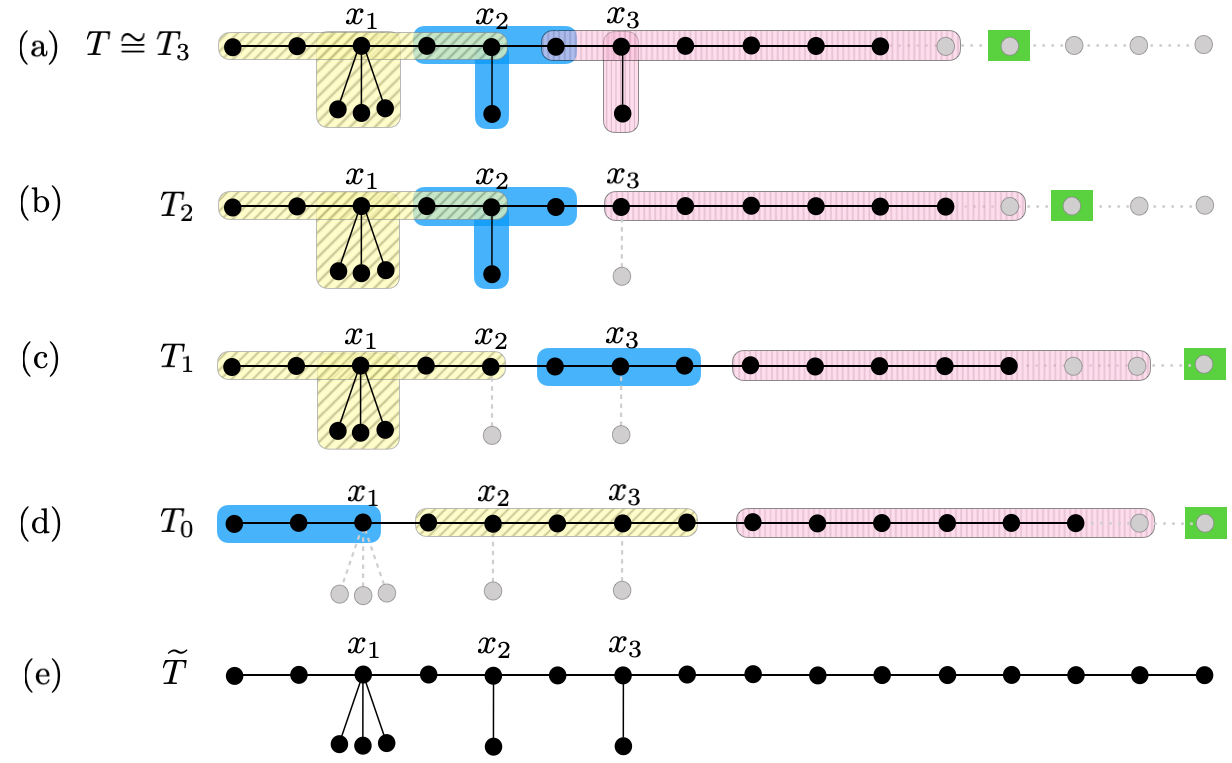} \]
\caption{A $1$-caterpillar $T$ and the sequence of $1$-caterpillars that form its cocoon.}

\label{fig:bigEx}
\end{figure}

The fundamental idea of our proof approach is to transform a cover of the path $P_n$ (where we know the BNC holds), into a cover of a $p$-caterpillar on $n$ vertices, thus implying the BNC for the $p$-caterpillar. We will use Figure~\ref{fig:bigEx} to illustrate this procedure. In parts (b-d) the grey vertices and dotted edges are not part of the $T_i$, $i=0,1,2,3$, but will be referenced in a later discussion.
Consider a $p$-caterpillar $T$ with $n$ vertices and $k$ roots;
for example part (a) of Figure~\ref{fig:bigEx} which shows a $1$-caterpillar with 16 vertices and 3 roots.
Our proof will begin with an $n$ vertex path and a cover with  
$\lceil \sqrt{n}\ \rceil$ balls (for example part (d)  of Figure~\ref{fig:bigEx} including the grey spine vertices displays $P_{16}$ and a cover with 4 balls). 
In our transformation we move vertices and edges from the right side of the path to the leftmost root. For example, starting with $T_0$ in part (d) of Figure~\ref{fig:bigEx}, to obtain $T_1$ in part (c) we remove one vertex from the right end of the spine and append three vertices to $x_1$.

 We note that the vertices in the subtree at this root may not be covered by the original cover on the path. For example, the spine cover in part (d) would not work for the graph in part (c), as the vertices appended to $x_1$ would not be covered. We will show, that under the right assumptions, we can rearrange the balls (neither adding nor removing balls) so that the subtree at the leftmost root is covered (for example, the arrangement of balls shown in part (c) results in the subtree rooted at $x_1$ being covered).  We continue the process, moving vertices and edges from the right side to place them at the next root, then rearranging the balls, so that the vertices at the root are covered (for example, in parts (b) and (a) of Figure~\ref{fig:bigEx} we can see this process continued in order to cover $T$). 

The outlined approach depends on having the correct conditions. Consider the example of Figure~\ref{fig:bad_idea}, where the right tree $T^*$ is transformed into the left tree $T$, in line with the above outlined process.
As pictured, $T^*$ has a  cover with 3 balls.
However, 3 balls do not suffice to cover $T$.
To see this, observe that if the ball of radius $2$ is not centered at $v_4$ or $v_5$, it will cover at most 7 vertices, leaving at least 5 vertices to be covered by balls of radius $0$ and $1$, which is impossible.  Thus, without loss of generality, the ball of radius $2$ must be centered at $v_4$.  
 However, centering it at $v_4$ leaves $v_1,v_7,v_8$, and the leaf adjacent to $v_6$ to be covered by the balls of radius $0$ and $1$, which is impossible.
The example shows that the process of transforming a cover of the path into a cover of a $p$-caterpillar by rearranging balls, may, in general, not work. If, in the middle of this process, we wanted to convert the cover of $T^*$ into a cover of $T$ (from Figure~\ref{fig:bad_idea}), this would not be possible using the same balls.  In the proof of Theorem~\ref{thm:catp}, we will make particular demands on the cover in order to successfully complete the process, using combinations  of \emph{shift} and \emph{jump} operations, as described in Section~\ref{secOps}.

\section{Covering Caterpillars}

As described in the previous section, our approach will involve successively transforming an existing cover of a $p$-caterpillar into a new cover of a slightly modified $p$-caterpillar, getting closer to our target $p$-caterpillar with each step.  We next discuss stages of caterpillars, covers, and then operations on covers.

\subsection{Caterpillar Cocoon}

When we refer to a $p$-caterpillar $T$, we standardize the notation to say it has spine $(v_1, \ldots, v_t)$ with roots
$x_1, \ldots, x_k$; this means that the spine consists of the path $(v_1, \ldots, v_t)$ and there are some 
$1< i_1 < \cdots < i_k < t$, such that  $x_j = v_{i_j}$ are the roots.  Let $\h_i$ be the height of the subtree at root $x_i$; observe $\h_i \leq p$.  Let $\nv_i$ be the number of vertices in the subtree rooted at $x_i$ (not counting the root itself): for example in Figure~\ref{fig:bigEx} (a), $\nv_1=3$, $\nv_2=\nv_3=1$ and $\h_1=\h_2=\h_3=1$.

We consider the indices increasing from $v_1$ on the left to $v_t$ on the right.
We use the next definition to encode the step-by-step process of transforming a path into a $p$-caterpillar.

\begin{defn}\label{defn:cocoon}
Suppose $T$ is a $p$-caterpillar with $k$ roots, $x_1, \ldots, x_k$, ordered from left to right.  The  \dword{cocoon} of $T$ is the sequence of $p$-caterpillars
\cocoon{T_0, T_1, \ldots, T_k}, where $T_k = T$, and $T_{i-1}$ is constructed from $T_i$ as follows: \begin{quote}
Consider $x_i$, i.e. the rightmost root of $T_i$. Remove all $\nv_i$ vertices in the subtree at $x_i$,
and extend $T_i$ at its rightmost spine vertex, by a path with
$\h_i$ vertices.
\end{quote}

\end{defn}

\noindent
Notice that $T_0$ is simply a path, which is in fact a $p$-caterpillar.  
When we have a $p$-caterpillar cocoon \cocoon{T_0, \ldots, T_k} of $T$, we will often want to refer to a graph \bigT \ (we call its \dword{wrapper}) which we define to be the tree $T$ with a path of $\sum_{i = 1}^k \nv_i$ vertices added to the rightmost spine vertex of $T$.  Then each $T_i$ is a subgraph of \bigT. 
To illustrate the last definition, observe that in Figure~\ref{fig:bigEx} 
 \cocoon{T_0, T_1, T_2, T_3} is the cocoon of $T$ and \bigT \ is its wrapper.
Recall, the dotted edges and grey vertices in (a)-(d)  do not exist in subgraphs $T_0,T_1,T_2,T_3$ but are used to illustrate the difference in spine vertices between the wrapper \bigT\ and the caterpillars in the cocoon.

\begin{defn}\label{defn:Ti} Consider a cocoon \cocoon{T_0, \ldots, T_k} of $p$-caterpillar $T$, and let \bigT\ be its wrapper.
A {\boldmath{$T_i$}}{\bf -cover in }{\boldmath{$\widetilde{T}$}} {\bf with excess} {\boldmath{$\varepsilon_i$}} is a spine cover of subgraph $T_i$ in \bigT\ such that a set of $\varepsilon_i$ vertices on the spine of $\bigT\ - T_i$ are covered.
We require that if a spine vertex of \bigT \  is covered, then any spine vertex to its left is also covered. 
\end{defn}

\noindent
Figure~\ref{fig:bigEx} illustrates the definition, where we see in part (d) a $T_0$-cover in \bigT with excess  $\varepsilon_0 = 2$, as the height of the subtree at $x_1$ is $1$ but the order is $3$.
Part (c) displays a $T_1$-cover in \bigT \ with excess $\varepsilon_1 = 3$, part (b) displays a 
$T_2$-cover in \bigT \ with excess $\varepsilon_2 = 2$, and part (a) displays
a $T_3$-cover in \bigT \ with excess $\varepsilon_3 = 2$.

Notice that when constructing $T_{i-1}$ from $T_i$ in the cocoon of $T$,
the number of vertices added to the right of the path is the height (i.e. $\h_i$) of the subtree, \emph{not} the number of vertices (i.e. $\nv_i$) in the subtree.  The rationale behind this distinction hinges on the fact that when using a spine cover, all that matters is the height of the subtree, so that extra vertices of a subtree that get covered (i.e. $\nv_i - \h_i$) make the $p$-caterpillar easier to cover, so we refer to them as excess.  
In later arguments, since we only consider spine covers, the hardest case is when the subtree rooted at each $x_i$ is just a path, so $\h_i = \nv_i$.  So, in this hardest case, 
when transitioning from $T_i$ to $T_{i-1}$, we are simply moving all the vertices of $T_i$ to the end of its spine.  So our definition of excess is set up to count this hardest case as having zero excess.

Recall that $k$ is the number of roots on the spine of $T$. We will typically start with the entire spine of \bigT \ covered, so for the initial excess  we count the number of spine vertices on \bigT \ that are not on $T_0$, i.e.
$$\varepsilon_0  = \sum_{i = 1}^k (\nv_i - \h_i) \ \geq 0.$$
Generally, we assume $\varepsilon_0 = 0$ because we consider the hardest case of  $\nv_i = \h_i$ (i.e. each subtree is merely a path),
though in one situation  we consider positive initial excess.
The goal of our proofs will be to arrive at a $T_k$-cover in \bigT\ with excess $\varepsilon_k \geq 0$,
since this immediately yields a cover of $T$.



\subsection{Basics of Covers}

A root $x$ of a $p$-caterpillar is \dword{tree-covered} by ball $B$
if $B$ is centered on the spine and all vertices of the subtree rooted at $x$ are in ball $B$. So if $B$ is of radius $r$ and centered
at $v_i$, then in addition to including some non-spine vertices, $B$ includes the $2r + 1$ spine vertices:
$v_{i-r}, \ldots, v_i, \ldots, v_{i+r}$; we refer to $v_{i-r}$ as the \dword{left endpoint}
of $B$ and $v_{i+r}$ as the \dword{right endpoint} of $B$. 
Consider a ball centered at $v_i$: if the ball has radius at least $p$ it is guaranteed to tree-cover $v_i$, otherwise it may not. Since this distinction will be highly relevant to our subsequent proofs, we say that a ball is \dword{tiny} if its radius is less than $p$ and \dword{non-tiny} if its radius is at least $p$. The \dword{tiny ball region} is the set of spine vertices that are covered by tiny balls.
If a spine-centered non-tiny ball $B$ contains $x_j$ but does not tree-cover it, there
are 2 ways this can happen:




\begin{enumerate}
\item If $x_j$ is left of the center of $B$ and $x_j$ is less than distance $\h_j$ from $v_{i-r}$, then we say $x_j$ is \dword{left-bad} in $B$.

\item If $x_j$ is right of the center of $B$ and $x_j$ is less than distance $\h_j$ from $v_{i+r}$ then we say $x_j$ is \dword{right-bad} in $B$.
\end{enumerate}

In the $T_0$-cover in \bigT\ in Figure~\ref{fig:bigEx} (d), vertex $x_1$ is right-bad.  In the $T_2$-cover in \bigT\ in Figure~\ref{fig:bigEx} (b), vertex $x_3$ is left-bad. 

\begin{defn}
 For a ball arrangement on the spine of a $p$-caterpillar, a sequence of balls $B_1, \ldots, B_t$ is \dword{special} (and \dword{starts at} $B_1$) if it has the following properties: 
\begin{enumerate}

\item \underline{Increasing}: $rad(B_i) < rad(B_{i+1})$ for any non-tiny balls $B_i$ and $B_{i+1}$. 

\item \underline{Non-overlapping}: $B_i$ is immediately left of $B_{i+1}$ (i.e. immediately after the right endpoint of
$B_i$, comes the left endpoint of $B_{i+1}$).

\item \underline{Tiny balls}: The tiny balls are contiguous and the tiny ball region contains no root vertices.

\item \underline{Cover}: All the spine vertices in $B_1$ and rightward are covered by some $B_i$.
\end{enumerate}
\end{defn}

Note that in a special cover, there could be other overlapping balls in the arrangement, which we are ignoring when
 selecting a particular special sequence.  For a ball arrangement on the spine of a $p$-caterpillar that contains ball $B$, we say the arrangement is 
\dword{special from} $B$ if \emph{there exists} a special sequence starting at $B$.
 A cover is simply \dword{special}, if it is special from the leftmost ball.  For example, in Figure~\ref{fig:bigEx} (d), the cover of $T_0$ in \bigT\ is special and in (c), the cover of $T_1$ in \bigT\ is special from ball $B^{[1]}$ and does not use the ball of radius 0.

We now describe the basic idea of the proof of Theorem~\ref{thm:catp} more carefully.
To denote a spine cover of $T$ (recall that in a \emph{spine cover} each ball is centered on the spine)  we will sometimes write
$$\cover{B_1, B_2, \ldots, B_s},$$
where the notation indicates that the center of ball $B_i$ is left of the center of ball $B_j$ when $i < j$; we refer to ball $B_i$ as being left of $B_j$. 
The idea of the proof is to start with a special cover
$\mathcal{C}_0$ of the spine of \bigT \ (for $N = \lceil \sqrt{n}\ \rceil$) i.e.
$$\cover{B^{[p]}, \ldots, B^{[N-1]}, B^{[p-1]}, \dots,  B^{[1]},B^{[0]} },$$
which is increasing by radius, except for the tiny balls at the right.  The spine of $\bigT$ has a burning number of $\lceil \sqrt{n}\ \rceil$ since it is a path of order $n$.
As we successively transform 
$T_0$ into $T_k = T$ (i.e. following the cocoon definition), we simultaneously transform the cover $\mathcal{C}_0$, 
which covers $T_0$ in \bigT, into a cover $\mathcal{C}_k$, which covers $T$ in \bigT.  Constructing the cover will just involve combinations of \emph{shift} and \emph{jump} operations (defined in the next subsection), so that at the end of the process, using the same balls, we will have a spine cover of $T$, proving our goal; namely that $\burn(T) \le \lceil \sqrt{n}\ \rceil$.

\subsection{Operations on Covers}
\label{secOps}

We now present the two basic operations we will perform on a cover.  Both operations take a cover for a $p$-caterpillar and rearrange some of the balls to ensure certain roots are tree-covered.  

\subsubsection{Shift Operation}

\begin{defn} {\bf (Shift Operation)}
Suppose we have the following cover of some $p$-caterpillar:
$$\cover{B_1, \ldots, B_{j-1}, B_j, \ldots, B_s}.$$
For an integer $\te \ge 1$, an \dword{$\te$-shift} at $B_j$ yields a new arrangement of balls, with the center of $B_i$, for $i \in \{j,j+1,\dots,s\}$, moved $\te$ vertices to the left on the spine.  The rest of the balls remain unmoved.

\label{defn:shift}\end{defn}

For example, in Figure~\ref{fig:bigEx} we see a 2-shift at $B^{[1]}$ in $T_1$ (part (c)), results in a cover of $T_2$ (part (b)); then a 1-shift at $B^{[3]}$,
results in a cover of $T_3$ (part (a)).  

We can see that the change in excess
is described by the equation:
$$\varepsilon_{j+1} = \varepsilon_j + \h_{j+1} -\te.$$
For example, $\varepsilon_1 = 3$ in Figure~\ref{fig:bigEx} (c), and 
the $2$-shift at $B^{[1]}$ (so $\te = 2$) used to tree-cover $x_2$, leads to a new excess (using $j = 1$)
of $\varepsilon_{2} = \varepsilon_1 + \h_{2} - \te = 3 + 1 - 2 = 2.$
Similarly, the 1-shift at $B^{[3]}$, results in excess
$\varepsilon_{3} = \varepsilon_2 + \h_{3} - 1 = 2 + 1 - 1 = 2.$
Note that shifting at ball $B_j$ will turn a non-overlapping arrangement of balls into an overlapping arrangement of balls, though such an arrangement will remain non-overlapping among the balls $B_j, B_{j+1}, \ldots$.


The \emph{Left Shift Lemma}  generalizes what happened in the example above when 
we did a 1-shift at $B^{[3]}$ in order to tree-cover $x_3$.  
In this case,
the root $x_3$ was left-bad in the ball that was shifted.  Since we have to shift
at most $\h_3$ to tree-cover $x_3$, the excess stays the same or increases.

\begin{lemma} \label{lem:leftShift} {\bf (Left Shift Lemma)}  Let \cocoon{T_0, \ldots, T_k} be a cocoon of $p$-caterpillar $T$ and let $\mathcal{C}_j$ be a $T_j$-cover in \bigT\ where $B$ is a non-tiny ball containing $x_{j+1}$.  Suppose \smallskip

{\rm \tiny \textbullet} $x_{j+1}$ is left-bad in $B$  

{\rm \tiny \textbullet} $\mathcal{C}_j$ is special from $B$, and

{{\rm \tiny \textbullet} root vertices to the left of the tiny ball region are at distance greater than $p$ from the 

\hspace{0.05in} tiny ball region.\smallskip

For some $\te \le \h_{j+1}$, an $\te$-shift at $B$ will produce a $T_{j+1}$-cover in \bigT\ that is special from $B$ with $\varepsilon_{j+1} \geq \varepsilon_j$.}\medskip

\end{lemma}

\begin{proof}
Since $\h_{j+1}$ is the height of the subtree rooted at $x_{j+1}$, there is an $\te \le \h_{j+1}$ for which an $\te$-shift at $B$ results in $x_{j+1}$ being tree-covered by $B$;   
essential to this point is that $\h_{j+1} \leq p \leq \rad(B)$.  Let $\mathcal{C}_{j+1}$ be the result of an $\te$-shift at $B$. 
Balls to the left of $B$ are not moved during the $\te$-shift, so roots $x_1,\dots,x_j$ remain tree-covered in $\mathcal{C}_{j+1}$.  Recall from Definition~\ref{defn:cocoon} the difference between $T_j$ and $T_{j+1}$: in $T_{j+1}$, subtree $T[x_{j+1}]$ is rooted at $x_{j+1}$, whereas in $T_j$, the 
subtree is removed and 
$\h_{j+1}$  vertices are appended as a path to the right end of the spine.  Since 
$\te \leq \h_{j+1}$, all the vertices of $T_{j+1}$ will be covered by $\mathcal{C}_{j+1}$. Furthermore, since we shift at most $p$,
the tiny ball region contains no root vertices.
 Thus, $\mathcal{C}_{j+1}$ is a $T_{j+1}$-cover in \bigT. Additionally, observe that 
$\mathcal{C}_{j+1}$ is special from $B$.  Finally, since $\te \leq \h_{j+1}$, $$\varepsilon_{j+1} = \varepsilon_j +\h_{j+1}-\te \geq \varepsilon_j.$$

\end{proof}

We state a weaker version of the Left Shift Lemma that removes the requirement of the cover being special. This will be useful in one part of proof of Theorem~\ref{thm:catp}.

\begin{lemma} \label{lem:WEAKleftShift} {\bf (Weak Left Shift Lemma)}  Let \cocoon{T_0, \ldots, T_k} be a cocoon of $p$-caterpillar $T$ and let $\mathcal{C}_j$ be a $T_j$-cover in \bigT\ where $B$ is a non-tiny ball containing $x_{j+1}$.  Suppose $x_{j+1}$ is left-bad in $B$. 
For some $\te \le \h_{j+1}$, an $\te$-shift at $B$ will produce a $T_{j+1}$-cover in \bigT\  with $\varepsilon_{j+1} \geq \varepsilon_j$.
\end{lemma}

The  \emph{Right Shift Lemma}  generalizes what happened in Figure~\ref{fig:bigEx}, part (c)  when 
we did a 2-shift at $B^{[1]}$ in order to tree-cover $x_2$ in part (b).  In this case,
the root $x_2$ was right-bad in a ball, so we leave that ball alone, and shift at the ball immediately to its right.  Now we might have to shift
up to $2 \h_2$ to tree-cover $x_2$, so the excess could go down as much as $\h_2$, so to keep the excess non-negative, we begin with excess at least $\h_2$.

\begin{lemma} \label{lem:rightShift} {\bf (Right Shift Lemma)} 
Let \cocoon{T_0, \ldots, T_k} be a cocoon of $p$-caterpillar $T$ and let $\mathcal{C}_j$ be a $T_j$-cover in \bigT\ where $B$ is a non-tiny ball containing $x_{j+1}$, and $B^*$ is a non-tiny ball directly right of $B$. If \smallskip

{\rm \tiny \textbullet} $x_{j+1}$ is right-bad in $B$, 

{\rm \tiny \textbullet} $\mathcal{C}_j$ is special from $B^*$, 

{\rm \tiny \textbullet} root vertices to the left of the tiny ball region are at distance greater than $2p$
from 

\hspace{0.075in} the tiny ball region, and

{\rm \tiny \textbullet} $\varepsilon_j \geq \h_{j+1}$,\smallskip

\noindent then for some $\te \le 2\h_{j+1}$, an $\te$-shift at $B^*$ will produce a $T_{j+1}$-cover in \bigT\ that is special from $B^*$ with $\varepsilon_{j+1} \geq \varepsilon_j-\h_{j+1} \geq 0$.
\end{lemma}

\begin{proof} 

For some $\te \leq 2\h_{j+1}$ an $\te$-shift at $B^*$ results in $x_{j+1}$ being tree-covered by $B^*$  (in the worst case, an $\h_{j+1}$-shift would get $x_{j+1}$ in $B^*$, but not tree-covered, requiring up to $\h_{j+1}$ additional shifts to tree-cover $x_{j+1}$); call this arrangement $\mathcal{C}_{j+1}$.  As the balls to the left of $B^*$ are not moved, roots $x_1,\dots,x_j$ were tree-covered in $\mathcal{C}_j$ and remain tree-covered in $\mathcal{C}_{j+1}$. 

Observe that there are $\h_{j+1}$ fewer vertices on the right end of the spine of $T_{j+1}$ than $T_j$.  From Definition~\ref{defn:Ti}, $\varepsilon_j$ is the number of vertices on the spine of \bigT $- T_j$ that are covered by $\mathcal{C}_j$.  As $\varepsilon_j \geq \h_{j+1}$, after we do an $\te$-shift at $B^*$ for $\te \leq 2\h_{j+1}$, 
$$\varepsilon_{j+1} = \varepsilon_j+  \h_{j+1}  -\te \geq \varepsilon_j+\h_{j+1} - 2\h_{j+1}= \varepsilon_j - \h_{j+1} \geq 0.$$  
Since $\varepsilon_{j+1} \geq 0$, all vertices on the spine of $T_{j+1}$ are covered by $\mathcal{C}_{j+1}$.  Furthermore, since we shift at most $2p$,
the tiny ball region contains no root vertices.
Thus, $\mathcal{C}_{j+1}$ is a $T_{j+1}$-cover in \bigT. Additionally, observe that $\mathcal{C}_{j+1}$ is special from $B^*$.\end{proof}


\subsubsection{Jump Operation}

\begin{defn} {\bf (Jump Operation)}
Suppose we have the following cover of some $p$-caterpillar:
$$\cover{B_1, \ldots, B_{i-1}, B_i, \ldots,  B_{k-1}, B_k, B_{k+1}, \ldots, B_s}.$$
To \dword{jump} $B_k$ to $B_i$ yields the following new arrangement of balls:
$$\cover{B_1, \ldots, B_{i-1}, B_k, B_i, \ldots,  B_{k-1}, B_{k+1}, \ldots, B_s},$$
i.e. $B_k$ has its left endpoint where $B_i$ used to have its left endpoint, and all the balls at $B_i$ and right, are moved to the right  so that the left endpoint of $B_i$ immediately follows the right endpoint of $B_k$. 
\end{defn}

The jump operation is pictured in Figure~\ref{fig:ob_twice} where we see that from the top illustration to the bottom illustration, ball $R$ is jumped to ball $L$.  The left endpoint of $L$, and any balls (whether zero or more) that were between $L$ and $R$, are shifted to the right by $2\cdot \rad(R)+1$ in order to accommodate the placement of ball $R$.
For an example of the jump operation, consider the cover of $T_0$ given in Figure~\ref{fig:bigEx} (d) and observe that $x_1$ is right-bad in $B^{[1]}$ and the cover in (d) has excess $\varepsilon_0=2$.  We jump $B^{[2]}$ to $B^{[1]}$ to get the cover of $T_1$ given in Figure~\ref{fig:bigEx} (c), and note that the resulting cover has excess 
$\varepsilon_1 = 3  = \varepsilon_0 + 1$.

 In general, in a cocoon, if we have a cover of $T_{j}$ in which $x_{j+1}$ is left- or right-bad  and we jump a ball to get a cover of $T_{j+1}$, then 
 $\varepsilon_{j+1} = \varepsilon_{j}+ \h_{j+1}$; we will see this point in the next lemma.
In the next lemma, we discuss when we can rearrange the balls (think of doing repeated jump operations) in order to tree-cover $z$ many roots, so our excess will increase by the sum of the corresponding $z$ many $\h_i$ values.  
This gain in excess will depend on the strong assumption of having enough non-tiny balls on the right.

\begin{figure}[h]
\[ \includegraphics[width=0.65\textwidth]{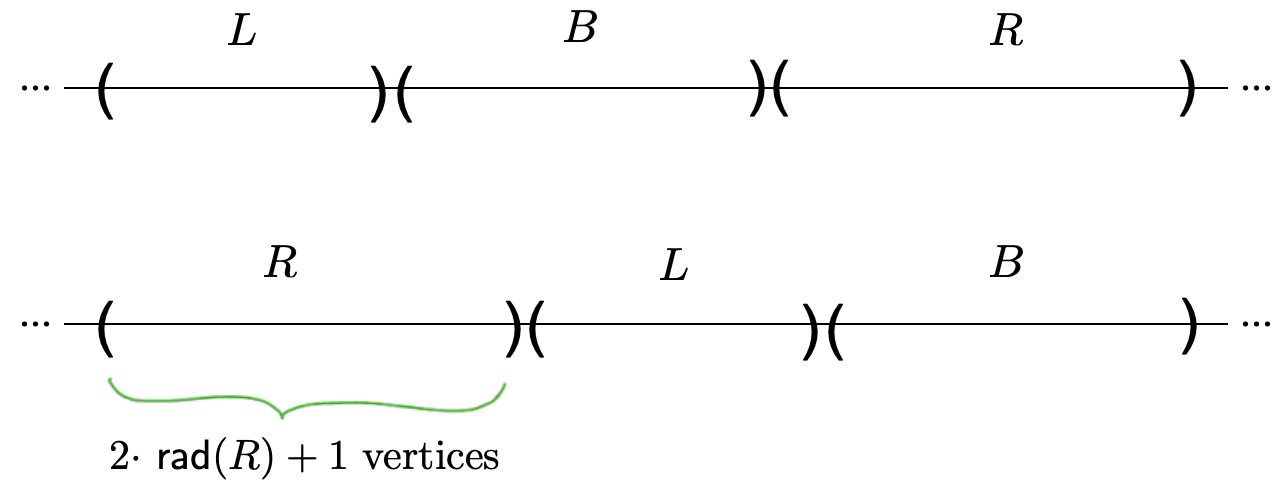} \]
\caption{An illustration of the jump operation with ball $R$ jumped to $L$.}
\label{fig:ob_twice}
\end{figure}


\begin{figure}[htbp]
\[ \includegraphics[width=0.875\textwidth]{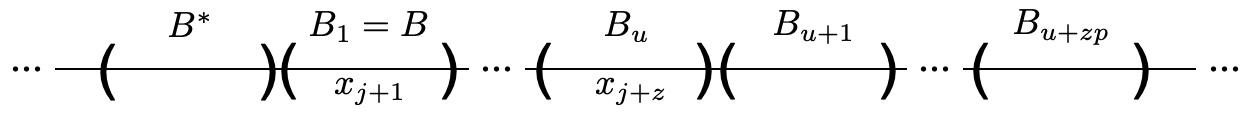}  \]

\caption{An illustration of the labelings of the balls described in the proof of Lemma~\ref{Pnew_lem_many_jumps}.}
\label{fig_Jump}
\end{figure} 

\begin{lemma} {\bf (Jump Lemma)}
\label{Pnew_lem_many_jumps} Let \cocoon{T_0, \ldots, T_k} be a cocoon of $p$-caterpillar $T$ and let $\mathcal{C}_j$ be a $T_j$-cover in \bigT, where $B$ is the rightmost non-tiny ball containing $x_{j+1}$, and $B^*$ is the ball directly left of $B$, if it exists. Suppose the tiny ball region is at the rightmost vertices of the spine. Suppose further that we are in one of the following two situations:\medskip

{\rm \tiny \textbullet} $x_{j+1}$ is right-bad or tree-covered in $B$ and $\mathcal{C}_j$ is special beginning at $B$, or \smallskip

{\rm \tiny \textbullet}  $x_{j+1}$ is left-bad in $B$ and $\mathcal{C}_j$ is special beginning at $B^*$. \medskip

\noindent Further, for some $z \geq 1$, suppose that there are at least $zp$ non-tiny balls of $\mathcal{C}_j$ right of the ball containing $x_{j+z}$.
There is a rearrangement of the balls that results in a $T_{j+z}$-cover in \bigT\ that is special beginning at the ball directly right of the ball containing $x_{j+z}$; and with excess
$$\varepsilon_{j+z} \geq \varepsilon_j + \sum_{i=j+1}^{z + j} \h_i.$$

\end{lemma}



\begin{proof} 
Since cover $\mathcal{C}_j$ is special from $B$ or $B^*$ we can select
a special sequence of non-overlapping balls starting at $B$, calling it:
\[B_1, B_2, B_3,\dots,B_u,B_{u+1},\dots,B_{u+zp},\dots,\]
where $B_1$ is $B$, $x_{j+z}$ is in $B_u$ (we allow for the possibility that $B_1=B_u$, i.e. $x_{j+z}$ is in $B$), the non-overlapping $B_i$ increase in radius (see Figure~\ref{fig_Jump}), and $B_i$ is non-tiny for $1 \leq i \leq u+zp$.

Note that it is important that we actually have non-tiny balls up to and including ball  $B_{u+zp}$ due to the assumption that we have at least $zp$ many non-tiny balls right of the root $x_{j+z}$.
In this proof, we refer to the roots $x_{j+1}, \ldots, x_{j+z}$ as the \dword{target roots}.  We will tree-cover all the target roots by a series of well chosen jump operations. 

The proof will involve a series of steps (at least one step, and at most $z$ steps, depending on the situation), where at each of these steps we will jump an appropriately chosen 
ball from $\mathcal{B}=\{ B_{u+ip} \ | \ i = 1, \ldots, z \}$ to a ball left of it.
A key point will be that the right boundary of the jumped ball $X$ from $\mathcal{B}$ is not \dword{problematic}, where by \emph{problematic} we mean that some target root is right-bad in $X$, or left-bad in the ball directly right of $X$,
i.e. there is a ``problem'' near the right boundary of $X$.

In Step 1, we consider target root $x_{j+1}$, which is in ball $B$.  If $x_{j+1}$ is 
right-bad we will jump an appropriate ball $X$ from $\mathcal{B}$  to the ball $B$,
while if $x_{j+1}$ is left-bad we will jump $X$ to the ball directly left of $B$ (in either case, tree-covering $x_{j+1}$).
In either case, we will choose $X$ to be the smallest ball from $\mathcal{B}$
such that  $X$ is not problematic.  We argue that such an $X$ exists.  First 
we note that any ball from $\mathcal{B}$ will work to tree-cover $x_{j+1}$,
because the radius of $X$ is at least $p$ larger than the ball it is jumped to;
thus $X$ occupies at least $2p$ more vertices.  Thus $x_{j+1}$ will
be in $X$ and neither left- nor right-bad. In other words, if $x_{j+1}$ is right-bad in $B$ then to guarantee that $x_{j+1}$ is not less than distance $\h_{j+1}$ from the right boundary we need only increase the radius of the ball that contains it by at most $p$, which is satisfied by jumping $X$ to $B$. If $x_{j+1}$ is left-bad in $B$, jumping $X$ to the ball directly left of $B$ results in $x_{j+1}$ being in $X$ and to the right of the center of the ball. It can be verified that $x_{j+1}$ is not right-bad in $X$ since as in the previous situation, the additional $p$ vertices guarantee that $x_{j+1}$ is not less than $\h_{j+1}$ from the right boundary of $X$.

We next argue that there is such a ball $X$ that will not be problematic for any target roots.  If $X = B_{u+p}$ is problematic for some target root, say $x_{j+2}$, then we instead consider $X = B_{u+2p}$, which will tree-cover both $x_{j+1}$ and $x_{j+2}$.  Continuing in this fashion, the step ends with some initial number of target roots (at least one) tree-covered, using a single ball from $\mathcal{B}$, leaving some (perhaps none) of the  target roots uncovered.  Suppose $x_{j+i}$ is the first uncovered target root; that is, $X$ covered the first $(i-1)$ target roots.  Since the jump was not problematic, the next target root $x_{j+i}$ starts as $x_{j+1}$ did: either it is i) right-bad (or already tree-covered), and special starting at the ball it is in, or ii) it is left-bad and special starting at the ball directly left of it.

Step 2 now does exactly the same kind of action as done in Step 1, though to the next target root $x_{j+i}$.  To see that the same reasoning applies, note that $x_{j+i}$ is somewhere among $B_1, \ldots, B_u$, so any choice of ball from 
$\mathcal{B}$
will tree-cover it. Furthermore, we can make an unproblematic choice,
since, as in Step 1, if a choice is problematic, we can go to the next bigger ball.
Note that in avoiding a problematic choice we cannot run out of balls from 
$\mathcal{B}$ since each time a ball is problematic and we go to the next bigger ball from $\mathcal{B}$, we in fact cover another target root.  

We continue in this manner using as many steps as needed to obtain a  $T_{j+z}$-cover in \bigT\ that is special, beginning at the ball directly right of the ball containing $x_{j+z}$.  
Finally, note that since we have covered roots just by rearranging balls, we have
 $\varepsilon_{j+z} \geq \varepsilon_j+ \sum_{i=j+1}^{j+z} \h_i$.

 \end{proof}

\section{Preliminary Results}

Hiller et al.\ \cite{cat2019} observed that for a $p$-caterpillar with spine length $t$, $\burn(T) \le \sqrt{t} + p$: first cover the vertices of the spine using balls of radius $0,1,\dots,\sqrt{t} - 1$; then to ensure all non-spine vertices will be covered, increase the radius of every ball by $p$.  Thus, the BNC holds trivially for $p$-caterpillars when $\sqrt{t} + p \le \sqrt{n}$.
Let $z$ be the number of non-spine vertices, so $n = t + z$, and thus, 
the trivial case is when $\sqrt{t+z} \ge \sqrt{t} + p$; squaring both sides of the inequality and simplifying yields the inequality $$z \ge 2p\sqrt{t} + p^2.$$

\begin{remark}\label{remark1} If the number of non-spine vertices is at least $2p\sqrt{t}+p^2$ then the BNC holds.\end{remark}

If we consider $p$ to be fixed, the result is trivial when the number of non-spine
vertices (i.e.~$z$) is somewhat larger than the square root of the number of spine vertices (i.e.~$t$).  
In Section~\ref{sec:pcats}, we consider case where the number of non-spine vertices is simply a polynomial of $p$, so as $t$ increases,
our results are not trivial.\medskip

In the following observation we address the issue of when the number of vertices in a $p$-caterpillar is not a perfect square.

\begin{observation}\label{obs:perfectsquare} 
If $(x - 1)^2 < y < x^2$ and the BNC holds for $p$-caterpillars on $x^2$ vertices, then the BNC holds for 
$p$-caterpillars on $y$ vertices.

 \end{observation}

\noindent 
The observation holds because for any $p$-caterpillar $T$ on $y$ vertices, we can extend its spine to obtain $p$-caterpillar $T^*$ on $x^2$ many vertices, and then easily transform a cover of $T^*$ with $x$ many balls into a cover of $T$ with 
 $\lceil \sqrt{y}\ \rceil = x$ many balls.

We note for the proof below, and subsequent proofs, that if non-overlapping, the tiny balls occupy 
$1+3+\dots+(2p-1) = p^2$ spine vertices. 
The next result points out that the most difficult case is when the subtrees of the roots are simply paths (or close to paths);  once the the number of vertices in the subtrees is large enough relative to the tree heights, the BNC follows in a straightforward manner, just using shift operations (no jump operations).  Thus, a good picture to have in mind is that the hardest case is when the subtrees are all paths.

\begin{prop}\label{theprop} The BNC holds for any $p$-caterpillar such that $\displaystyle\sum_{i=1}^k (\nv_i-2\h_i) \geq p^2 + p$.  

\end{prop}

\begin{proof} 

Suppose a $p$-caterpillar $T$ on $n$ vertices has cocoon \cocoon{T_0, \ldots, T_k}, where by Observation~\ref{obs:perfectsquare}  we can assume $n$ is a perfect square.  Begin with a special $T_0$-cover in \bigT \ which covers the entire spine of \bigT; since $n$ is a perfect square we can achieve the required non-overlapping condition.
Then by the assumption we start with initial excess
$$\varepsilon_0 = \sum_{i=1}^k (\nv_i - \h_i) \geq p^2 + p + \sum_{i=1}^k \h_i.$$  
Proceeding by induction, assume that for some $j \geq 0$, we have a $T_j$-cover in \bigT, call it $\mathcal{C}_j$, that is special beginning
at a ball containing $x_j$ (except for $j=0$, it is special from the leftmost ball),
and has excess   
$$\varepsilon_j  \geq p^2  + p + \sum_{i=j+1}^k \h_i.$$  

Now consider $x_{j+1}$, and we will turn $\mathcal{C}_j$ into a cover $\mathcal{C}_{j+1}$ of $T_{j+1}$.
If $x_{j+1}$ is root-covered by the $T_j$-cover in \bigT, we let $\mathcal{C}_{j+1} = \mathcal{C}_j$. 
If  $x_{j+1}$ is left-bad, we do an $\te$-shift for $\te \leq \h_{j+1}$ (using the Left Shift Lemma) to get $\mathcal{C}_{j+1}$.
If $x_{j+1}$ is right-bad we do an $\te$-shift  for $\te \leq 2\h_{j+1}$ (using the Right Shift Lemma) to get $\mathcal{C}_{j+1}$.
The next paragraph justifies the application of either Shift Lemma.
Observe that $\mathcal{C}_{j+1}$ is special beginning at the leftmost ball containing $x_{j+1}$, and in the worst case we get the required inductive conclusion:
$$\varepsilon_{j+1} \geq \varepsilon_j +\h_{j+1} - 2\h_{j+1} \ge p^2 + p + \sum_{i=j+2}^k \h_i.$$  

We now justify our applications of the shift lemmas.  
We use the fact that the excess is always at least $p^2 + p$.
The Left Shift Lemma requires the ball containing the root (i.e. $x_{j+1}$)
to be non-tiny.  Since the tiny balls occupy $p^2$ vertices, the excess being at least $p^2$ is sufficient to keep any root safely to the left of any tiny ball.
The Right Shift Lemma (applied to a right-bad root) makes the stronger requirement that ball right of the ball containing the root is non-tiny.  
In this case, the root is right-bad in some ball $B$, so it is at most distance $p-1$ from the right endpoint of $B$.  The ball right of $B$ must be non-tiny, otherwise there would be at most $p^2+p-1$ vertices right of the root, which is insufficient to account for our excess of $p^2+p$.

\end{proof}

In the last proposition we see that one of the key points is that the tiny balls remain on the right of the spine, not interacting with the roots.  We do this by keeping the excess up to at least $p^2+p$.
As a bigger point, we can see Remark~\ref{remark1} and Proposition~\ref{theprop} highlighting two ways the BNC becomes straightforward.
Remark~\ref{remark1} points out that if there are enough non-spine vertices, the BNC is easy.
Proposition~\ref{theprop} points out that if the subtrees at the roots are sufficiently dense, then the BNC follows.  

Section~\ref{sec:pcats} considers the harder case where we do not have enough non-spine vertices, and the subtrees are sparse.
Given the complication caused by the tiny balls, it is tempting to dispense with them, and just add one single large ball.  However, then our main result would not prove the BNC, since the tiny balls are sometimes required in order to start the process off with a cover of the spine of \bigT.

\section{The main result: $p$-caterpillars}\label{sec:pcats}


In this section, we prove the burning number conjecture holds for all $p$-caterpillars with a sufficient number of vertices. We begin with a technical lemma and a rough overview of the proof of Theorem~\ref{thm:catp}, before proceeding with the actual proof of Theorem~\ref{thm:catp}.

 At various places in this section we will refer to the left or right half of a path, and say that some
balls are entirely contained in one of the halves - by this we mean that all of the vertices in such a ball are contained in that half (not just its center or some of its vertices).
In the next technical lemma 
 we lower bound the number of balls in the right half of a path.  

\begin{lemma} \label{lem_ball_calc}
For positive integers $z$ and $p$, if $n > 16z^2p^2$ and $n$ is a perfect square, then any special cover of an $n$ vertex path has at least $zp$ non-tiny  
balls entirely contained in the $\lfloor n/2 \rfloor$ rightmost vertices. 
\end{lemma}

\begin{proof}
Let $N = \sqrt{n} \geq 4zp+1$.
Since the cover is non-overlapping and the largest ball has radius $N-1$, the number of vertices covered by the largest $zp$ balls is 

$$\sum_{i=N-zp}^{N-1}(2i+1) = N^2 - (N-zp)^2 = 2zpN-z^2p^2.$$

\noindent
There are $\lfloor n/2\rfloor -p^2 = \lfloor N^2/2\rfloor - p^2$ vertices among the rightmost $\lfloor n/2\rfloor $ vertices of the path that must be covered by non-tiny balls.  
Thus, to guarantee at least $zp$ non-tiny balls in the right half of the path, it suffices to show $\lfloor N^2/2\rfloor - p^2 \ge 2zpN - z^2p^2$, 
which follows from
$N(N-4zp)+2p^2(z^2-1) - 1 \ge 0$ (the ``$-1$'' arises in the case where $n$ is odd).  The inequality holds because $z \geq 1$ and $N \geq 4zp+1$.\end{proof}

We now describe the rough approach to the proof of Theorem~\ref{thm:catp}.  
The Jump Lemma is initially applied to cover enough roots on the left side of the path $T_0$ so that the excess is sufficiently large.  Then, if there are few roots (Case 1), we repeatedly apply the shift lemmas to cover the rest of the roots.  If there are many roots (Case 2),
then we carry out a process which successively tree-covers the roots, from left to right.  If a root is left-bad, we can apply the Left Shift Lemma, and this action does not decrease the  excess.
If a root is right-bad, simply applying the Right Shift lemma could reduce the excess, so instead we execute ``$p$ jumps'' (i.e. apply the Jump Lemma with $z = p$) and then one shift.  At the end, the excess is sufficiently large to ensure we can shift to tree cover the final few roots.
We now formally prove Theorem~\ref{thm:catp}.

\bigskip
\noindent {\bf Theorem~\ref{thm:catp}.} {\it The BNC holds for any $p$-caterpillar on at least $16(4p^3 + 2p^2 + 4p)^2p^2$ vertices.}





\begin{proof}

We assume $p > 1$ as it is already known \cite{cat2019}, \cite{Liu} that the BNC holds for caterpillars (i.e. 1-caterpillars).
Let $T$ be a $p$-caterpillar on $n \ge 16(4p^3 + 2p^2 + 4p)^2p^2$ vertices. 
By Observation~\ref{obs:perfectsquare}, we can assume $n$ is a perfect square. 
We have chosen $n$ to satisfy two key inequalities which we call
Inequality~\ref{ineqGiven} and Inequality~\ref{ineqCalc}.

\begin{enumerate}

\item \label{ineqGiven}
$n \ge 16(4p^3 + 2p^2 + 4p)^2p^2$ 

\item  \label{ineqCalc}
$\dfrac{n - 2(4p^3 + 2p^2 + 4p)}{4(4p^3 + 2p^2 + 4p)} \ge \sqrt{n} + p^2$

\end{enumerate}
Inequality~\ref{ineqGiven} holds by assumption.
To see why Inequality~\ref{ineqCalc} holds, 
let $x = 2(4p^3+2p^2+4p)$ and observe that Inequality~\ref{ineqCalc} can be expressed as 
$$\frac{n-x}{2x} \geq \sqrt{n} + p^2 
~\Longleftrightarrow~ \sqrt{n}(\sqrt{n}-2x) \geq (2p^2+1)x.$$
\noindent
We are given that $n \geq 4x^2p^2$, so $\sqrt{n} \geq 2xp$, so since $p>1$ we have: 
$$\sqrt{n}(\sqrt{n}-2x) \geq 2xp(2xp-2x) = 4x^2p(p-1) \geq 4x^2p = (32p^4+16p^3+32p^2)x \geq (2p^2+1)x.$$  Therefore Inequality~\ref{ineqCalc} holds. 

Suppose $T$ has $r$ roots and let  \cocoon{T_0, \ldots, T_r} be a cocoon of $T$.  Let $\mathcal{C}_0$ be a special $T_0$-cover in \bigT\ which covers the entire spine of \bigT, where the tiny ball region is at the right.  
We proceed by cases based on  $r$.
In both cases we describe a process which begins with a number of jumps and ends with a number of shifts. It will be essential that we never shift at a tiny ball (required by the shift lemmas, since the tiny balls cannot  always tree-cover a root). At the end of each case, we show that the excess is sufficient to avoid the tiny balls.  

\noindent
{\bf Case 1: {\boldmath $T$} has {\boldmath $r < 4p^3 + 2p^2 + 4p$} roots. }

\noindent
We can assume 
there is some $b<r$ so that
$x_1, \ldots, x_b$ are among the leftmost $\lceil n/2\rceil $ vertices on the spine of \bigT\ and
$\sum_{i=1}^b \h_i \ge \sum_{i=b+1}^r \h_i$, i.e. think of the first half as
``heavier''  
(if the second half is heavier, we can just reverse the meaning of left and right in $T$).
Using Inequality~\ref{ineqGiven} with $z = b < 4p^3 + 2p^2 + 4p$, we apply Lemma~\ref{lem_ball_calc} in order to conclude that
there are at least $p b$ non-tiny balls entirely among the rightmost $\lfloor n/2 \rfloor$ vertices of the spine of \bigT.  
Thus, we can apply the Jump Lemma (with $z = b$ and $j = 0$) to
the first  $b$ roots of $T$, to conclude there is a rearrangement of the balls that results in a
$T_b$-cover in \bigT, called $\mathcal{C}_b$, with excess 
$\varepsilon_b \ge \sum_{i=1}^b \h_i$.  Furthermore, the Jump Lemma ensures $\mathcal{C}_b$ is special from the ball directly right of the ball containing $x_b$.   
At this point, the idea is to repeatedly shift in order to tree-cover the rest of the roots.  However we have to take some care because  if there are roots close to the right end of the spine, they would be among the tiny balls, where the shift lemmas are invalid; we will get around this by moving all the small balls to a region sparse in roots.

The right half has fewer than $r$ roots that break up the right half of the spine into root-free intervals (i.e. maximal sets of consecutive non-root spine vertices).  
Since we have 
at least $\frac{n}{2} - r = \frac{n - 2r}{2}$ non-roots among these, we must have some root-free interval of size at 
least
$$I =  \dfrac{\frac{n-2r}{2}}{r} = \dfrac{n-2r}{2r} \ge 2\sqrt{n} + 2p^2.$$  The last inequality holds by Inequality~\ref{ineqCalc} using $r \leq 4p^3 + 2p^2 + 4p$. 

We break up this interval into its first half of $\lceil I/2 \rceil$ vertices and second half of $\lfloor I/2 \rfloor$ vertices.
If there is a ball that contains vertices from both halves, call it $B$.  Even if $B$ were the largest ball, it occupies at most
$\sqrt{n}$ vertices in the second half.  Let $v$ be the vertex immediately right of $B$, or simply the leftmost vertex of the second
half if there is no $B$.  In either case, $v$ is the leftmost vertex in some ball in the cover.  Move all the tiny balls to start at $v$ and shift the balls at $v$ or right, to the right.   Since the tiny balls only cover $p^2$ vertices, 
 and 
$\lfloor I/2 \rfloor \ge \sqrt{n} + p^2$, we can fit all the tiny balls into the right half of this interval. 

In the current situation we have covered roots $x_1, \ldots, x_b$.  Suppose $x_b$ is in ball $B$ and $R$ is the ball immediately right of $B$.  The cover is special from ball $R$.  We will now carry out shifting in order to tree-cover the roots in ball $B$ (we will repeat a similar process at two later points in this proof, referring to this process as \dword{clearing the roots}).
First, for any roots that are left bad in $B$, by the Weak Left Shift Lemma, we can shift at $B$ so that none of these are left bad in $B$; this will not decrease the excess.  If there are now no left-bad or right-bad vertices in $B$, then the cover is special starting at $R$ as desired.
Second, if any of 
 roots are 
 right-bad in $B$ then we shift at $R$ till $R$ contains them all, and none of these roots remain left-bad in $R$ (this is possible since all these roots are within distance $p$ of the right endpoint of $B$).  Such shifting does not decrease the excess, so the worst case is actually if the only root in $B$ is $x_{b + 1}$ and it is right-bad in $B$.  In this case, we apply the Right Shift Lemma to $x_{b + 1}$; that is, we do an
 $\te$-shift for some $\te \leq 2h_{b+1}$.  In this worst case, this results in a $T_{b + 1}$-cover in \bigT, called 
 $\mathcal{C}_{b + 1}$, with excess 
$\varepsilon_{b + 1} \geq \varepsilon_{b} - h_{p+1}$
and which is special starting at the ball containing $x_{b + 1}$.

We next repeatedly apply the Left Shift Lemma or Right Shift Lemma, doing a shift for each of the roots in the right half.  This results in a $T_r$-cover in \bigT \ with excess: 
$$\varepsilon_r \geq  
\varepsilon_b - \sum_{i=b + 1}^{r} \h_i  \geq
\sum_{i=1}^b \h_i - 
\sum_{i=b + 1}^{r} \h_i \ge 0.$$

In order for the prior shifting to work, it is essential that we never shifted at a tiny ball. Each shift is at most a $2p$-shift, so in total we shift at most 
$(4p^3 + 2p^2 + 4p)2p \le \sqrt{n},$ where the last inequality follows from Inequality~\ref{ineqGiven}.  Since any tiny ball is at least $\sqrt{n}$ to the right of any root, we never in fact shift at a tiny ball.  
\medskip

\noindent
{\bf Case 2: {\boldmath $T$} has {\boldmath $r \ge 4p^3 + 2p^2 + 4p$} roots.} 

\noindent
 In order to simplify some discussion, let $\Q = 2p^3 + p^2 + 2p$.
Since the spine of \bigT \ is a path on $n$ vertices, without loss of generality, we can assume there are
$\lceil r/2\rceil \ge \Q$ roots among the first $\lceil n/2 \rceil$ vertices of 
the spine of
\bigT, 
(if the first half has fewer roots than the second half, just reverse the meaning of left and right).
 During the initial phase, we will apply the Jump Lemma to the first 
 $\Q$ roots $x_1,\dots,x_{\Q}$, and the Right Shift Lemma to 
 $x_{\Q + 1}$.  
 Then we will proceed by a process of combining jumps and shifts for the roots until we get low on balls, at which point we will shift everything remaining; we now give the details.

 We begin the initial phase.
We apply
Lemma~\ref{lem_ball_calc} with $z = \Q$ (using above Inequality~\ref{ineqGiven})
to conclude that we have at least
 $\Q p$ non-tiny balls entirely in the rightmost $\lfloor n/2 \rfloor$ vertices of 
 the spine of \bigT.
Since we have at least $\Q p$ non-tiny balls to the right of the ball containing $x_{\Q}$, 
we  can apply the Jump Lemma with $z=\Q$ and $j=0$.
  Thus, there is a rearrangement of the balls so that the result is a $T_{\Q}$-cover in \bigT, called $\mathcal{C}_{\Q}$, with excess 
  $\varepsilon_{\Q} \ge \Q$ 
  (since in the worst case we start with excess 0, and 
  the $p$ jumps happen on roots whose trees are simply single edges, i.e. the 
  $\nv_i = \h_i  = 1$).

 Suppose  $x_{\Q}$ is now in ball $B$ and $R$ is the ball directly right of $B$;
 the Jump Lemma ensures $\mathcal{C}_{\Q}$ is special from $R$.  We want to start the inductive process at $R$, but first need to deal with other roots that might be in $B$ (i.e. from among $x_{\Q + 1}, x_{\Q + 2}\ldots)$. 
 We deal with these
as we did above when we ``cleared the roots.'' Like above we consider the worst case where the next root $x_{\Q+1}$ is the only root we cover, and requires the Right Shift Lemma. So we do an $\te$-shift for some $\te \leq 2p$.  This results in a $T_{\Q + 1}$-cover in \bigT, called 
 $\mathcal{C}_{\Q + 1}$, with excess 
$\varepsilon_{\Q + 1} \geq \varepsilon_{\Q} - p \ge \Q - p$
and which is special starting at the ball containing $x_{\Q + 1}$.

The initial phase is now complete.
The key point is that the excess is large enough, so that at the end, 
when we do multiple shifts, the tiny balls will not be involved in the shifts (details below).
We start the inductive process at $j = \Q + 1$ (i.e. assuming the worst case start).  Suppose we have a $T_j$-cover in \bigT, called $\mathcal{C}_j$, with excess at least 
$\Q - p$, which is special starting at the ball containing $x_j$.  
We have cases.

If $x_{j+1}$ is left-bad in its ball then we apply the Left Shift Lemma.  That is, we do an $\te$-shift for $\te \leq \h_{j+1}$ 
to obtain a $T_{j+1}$-cover $\mathcal{C}_{j+1}$ in \bigT\ with excess 
$\varepsilon_{j+1} \geq \varepsilon_j \geq \Q - p$.  By the Left Shift Lemma, the $T_{j+1}$-cover in \bigT \ is special, starting at the ball containing $x_{j+1}$.
In the easy case that $x_{j+1}$ is not bad (neither left nor right), just let $\mathcal{C}_{j+1}=\mathcal{C}_j$, so that
it is a $T_{j+1}$-cover in \bigT\ with excess $\varepsilon_{j+1} \geq \varepsilon_j \geq \Q - p$.

If $x_{j+1}$ is right-bad in ball $B$ then we apply the Jump Lemma (with $z=p$), followed by the Right Shift Lemma.  We use $z = p$ because in the worst case, the next $p$ roots are roots of single edges.   To apply the Jump Lemma, there must be at least $zp = p^2$ balls to the right of $x_{j+p}$;  for now, we assume we have enough balls.  Applying the Jump Lemma results in a $T_{j+p}$-cover in \bigT\ where $x_{j+p}$ is now in ball $B^*$. Note that the $T_{j+p}$-cover in \bigT\ has excess $\varepsilon_{j+p} =\varepsilon_j + p \geq \Q - p +p = \Q$ and is special starting at the ball directly to the right of the ball $B^*$ containing $x_{j+p}$. 
At this point we may have a number of roots in $B^*$.  We deal with these
as we did above when we ``cleared the roots.'' Like above we consider the worst case where the next root $x_{j+p+1}$ is the only one to deal with, and requires the Right Shift Lemma.
Observe that $x_{j+p+1}$ could be right-bad in $B^*$ or a ball further to the right; or could be left-bad in a ball to the right of $B^*$.  
We apply the appropriate shift lemma in order to do an $\te$-shift for $\te \leq 2p$ resulting in a $T_{j+p+1}$-cover in \bigT\ with excess $\varepsilon_{j+p+1} \geq \varepsilon_{j+p} -p \geq \Q - p$.  By the shift lemmas, the $T_{j+p+1}$-cover in \bigT\ is special, starting at the ball containing $x_{j+p+1}$.\medskip 

Now consider the case above when the root is right-bad, but we have less than $p^2$ balls to the right.  
That is, 
for some $j$, we have a $T_j$-cover in \bigT,  with excess 
$\varepsilon_j \geq \Q - p$ 
which is special starting at the ball containing $x_j$.  Furthermore, suppose the number of (non-tiny) balls to the right of the $B$ (the ball containing $x_{j+1}$) is fewer than $p^2$.  We can use $2p$-shifts to cover the vertices that are right or left-bad in each remaining ball (including $B$, these right-bad vertices are in at most $p^2$ balls) in order to tree-cover all remaining roots, leaving excess at least 
$$\Q - p -2p(p^2) = 2p^3 + p^2 + 2p - p - 2p^3 = p^2 + p.$$  

\noindent
This excess guarantees that in our shifting we never shift at a tiny ball (which would be problematic).  When we shift we look at the ball the current root is in, or consider the ball immediately to its right. The tiny balls occupy $p^2$ vertices, so they are all right of any current root.  The excess of $p^2 + p$ (i.e. with an extra $p$) also ensures that the ball immediately to the right of any current root is not tiny.

\end{proof}

\medskip

Although Theorem~\ref{thm:catp} proves that the BNC holds for sufficiently large $p$-caterpillars, we expect that our bound on $n$ could be significantly improved. Theorem~\ref{thm:catp} proves there is a spine cover for any sufficiently large $p$-caterpillar, which is overkill for proving the BNC.  We also note that in many cases, the number of balls required for the Jump Lemma is excessive.

\begin{quest} Can the requirement on the number of vertices of a $p$-caterpillar from Theorem~\ref{thm:catp} be reduced? 
\end{quest} 

The BNC holds for $1$-caterpillars and $2$-caterpillars with no restrictions on $n$.  We conjecture that the shift and jump operations could be used to prove the BNC holds for $p$-caterpillars for other small values of $p$, such as $p=3$ or $p=4$.

\begin{quest} Can our proof method be modified in order to prove the full BNC for $p$-caterpillars, for some small values of $p$?
\end{quest}

\end{document}